\documentclass[letterpaper, 10 pt, conference]{ieeeconf}  

\usepackage[cmex10]{amsmath}
\usepackage[pdftex]{graphicx}
\usepackage{epsfig}
\usepackage{amssymb}
\usepackage{color}
\usepackage{graphics}
\usepackage{empheq}
\usepackage{hyperref}
\usepackage{array}
\usepackage{cite}
\usepackage{bm}

\newtheorem{prop}{Proposition}

\newtheorem{remark}{Remark}
\newtheorem{lem}{Lemma}

\def\f{\frac}

\def\i1n{i=1,\cdots,n}
\def\j1n{j=1,\cdots,n}
\def\ij1n{i,j=1,\cdots,n}

\newcommand{\be}{\begin{equation}}
\newcommand{\ee}{\end{equation}}
\newcommand{\beq}{\begin{equation*}}
\newcommand{\eeq}{\end{equation*}}
\newtheorem{thm}{Theorem}

\IEEEoverridecommandlockouts                              

\title{\LARGE \bf
Backstepping Control of the One-Phase Stefan Problem
}

\author{Shumon Koga, Mamadou Diagne, Shuxia Tang, and Miroslav Krstic
\thanks{Shumon Koga, Shuxia Tang and Miroslav Krstic are with the Department of Mechanical and Aerospace Engineering, U.C. San Diego, 9500 Gilman Drive, La Jolla, CA, 92093-0411, {\tt\small skoga@ucsd.edu}, {\tt\small mdiagne@ucsd.edu}, {\tt\small sht015@ucsd.edu}, and {\tt\small krstic@ucsd.edu}}
\thanks{Mamadou Diagne is with Mechanical Engineering Department of the University of Michigan Ann Arbor, MI 48109-2102, USA,
       {\tt\small mdiagne@umich.edu}}
}

\begin{document}

\maketitle
\thispagestyle{empty}
\pagestyle{empty}


\begin{abstract}
In this paper, a backstepping control of the one-phase Stefan Problem, which is a 1-D diffusion Partial Differential Equation (PDE) defined on a time varying spatial domain described by an ordinary differential equation (ODE), is studied. A new nonlinear backstepping transformation for moving boundary problem is utilized to transform the original coupled PDE-ODE system into a target system whose exponential stability is proved. The full-state boundary feedback controller ensures  the exponential stability of the moving interface to a reference setpoint and the ${\cal H}_1$-norm of the distributed temperature by a choice of the setpint satisfying given explicit inequality between initial states that guarantees the physical constraints imposed by the melting process. 
\end{abstract}

\section{INTRODUCTION}

Diffusion PDEs  with moving boundaries have been studied actively for the last few decades, and their understanding continues to be of high interest due to their extent for various industrial processes. Representative applications include sea-ice melting and freezing\cite{wett91}, continuous casting of steel \cite{petrus2012}, crystal-growth \cite{conrad_90}, and thermal energy storage system\cite{Belen03}.

 While the numerical analysis of  these  systems  is widely covered in the literature, their control related problems have been addressed relatively fewer. In addition to it, most of the proposed     control approaches are based on finite dimensional approximations with the assumption of  an explicit and known moving boundary \cite{Daraoui2010},\cite{Armaou01},\cite{Petit10}. For instance, diffusion-reaction processes with explicitly known moving boundaries are investigated in  \cite{Armaou01} based on the concept of inertial manifold \cite{Christofides98_Parabolic} and the  partitioning of the infinite dimensional dynamics into slow and fast  finite dimensional modes. We also refer the reader to \cite{Petit10} for the  motion planning boundary
control  of a one-dimensional one-phase nonlinear Stefan problem based on series representation  which leads to highly complex solutions that reduce the design possibilities. Instead of controlling the position of the liquid-solid interface by means of the temperature or the heat flux at the boundary, \cite{Petit10}  solves  the inverse problem  assuming  a prior known  moving boundary which help to derive the manipulated input.

 More complicated  approaches that lead to significant mathematical complexities in the process characterization are developed  based on an infinite dimensional framework. These methods  impose a number of constraints on the initial data and the state variables to achieve convergence of the dynamical coupled systems to the desired equilibrium. In order to avoid surface cracks during the solidification stage in a steel casting process represented by a diffusion PDE-ODE system defined on a time-varying spatial domain, an enthalpy-based Lyapunov functional  is used in  \cite{petrus2012}   to  ensure stabilization of  the temperature and the moving boundary at the desired setpoint. The aftermentioned results are derived restricting a priori, the input signal to be strictly positive and the moving boundary to be a non-decreasing function of time. In \cite{maidi2014} a geometric control approach \cite{karvaris1990-a,karvaris1990-b,maidi2009}  that enables the manipulation of  the boundary  heat flux to adjust the position of a liquid-solid interface at a desired setpoint is proposed, and the exponential stability of the ${\cal L}_2$-norm of the distributed temperature is ensured using a Lyapunov analysis. 

In this paper, the backstepping control \cite{krstic2008boundary,andrew2004} of a one phase Stefan problem   \cite{petrus2012,petrus2010, maidi2014} is studied. The exploitation of the  PDE  backstepping methodology for moving boundary problems is not well  investigated in the literature. An extension of  the backstepping-based observer design to the state estimation of parabolic PDEs with time-dependent spatial domain was proposed in  \cite{Izadi15} with an application to the well-known Czochralski crystal growth process. The authors  offer  a design tool  for the stabilization of an unstable parabolic PDE system with moving interface using  a collocated boundary measurement and actuation located at the moving boundary. The novelty of our approach relies to the  proposition of a new nonlinear  backstepping transformation that allows to deal with moving boundary problem without the need to rescale the system into a fixed domain. Our design of backstepping transformation for moving boundary stands as an extension of the one proposed in \cite{krstic2009, Gian2011, tang2011state} for linear PDE-ODE defined on a fixed spatial domain. The proposed controller enables the exponential stability of sum of the moving interface and the ${\cal H}_1$-norm of the  temperature profile under physical constraints which restrict the choice of reference setpoint with respect to the initial data. 

This paper is organized as follows: The one-phase Stefan problem  is presented  in Section \ref{model} and the control related problem in Section \ref{statement}. Section  \ref{nonlineartarget} introduces  the  nonlinear backstepping transformation for moving boundary problems. The properties of the proposed backstepping  controller are stated in Section \ref{positivness} and the Lyapunov stability of the closed-loop system under full-state feedback is established in Section \ref{stability}  with the exponential convergence of ${\cal H}_1$-norm of the distributed temperature and the moving boundary to the desired equilibrium.  Supportive numerical simulations are provided in Section \ref{simulation}. The paper ends with final remarks and future directions discussed in  Section \ref{conclusion}.

\section{Description of the Physical Process}\label{model}
\begin{figure}[htb]
\centering
\includegraphics[width=3.2in]{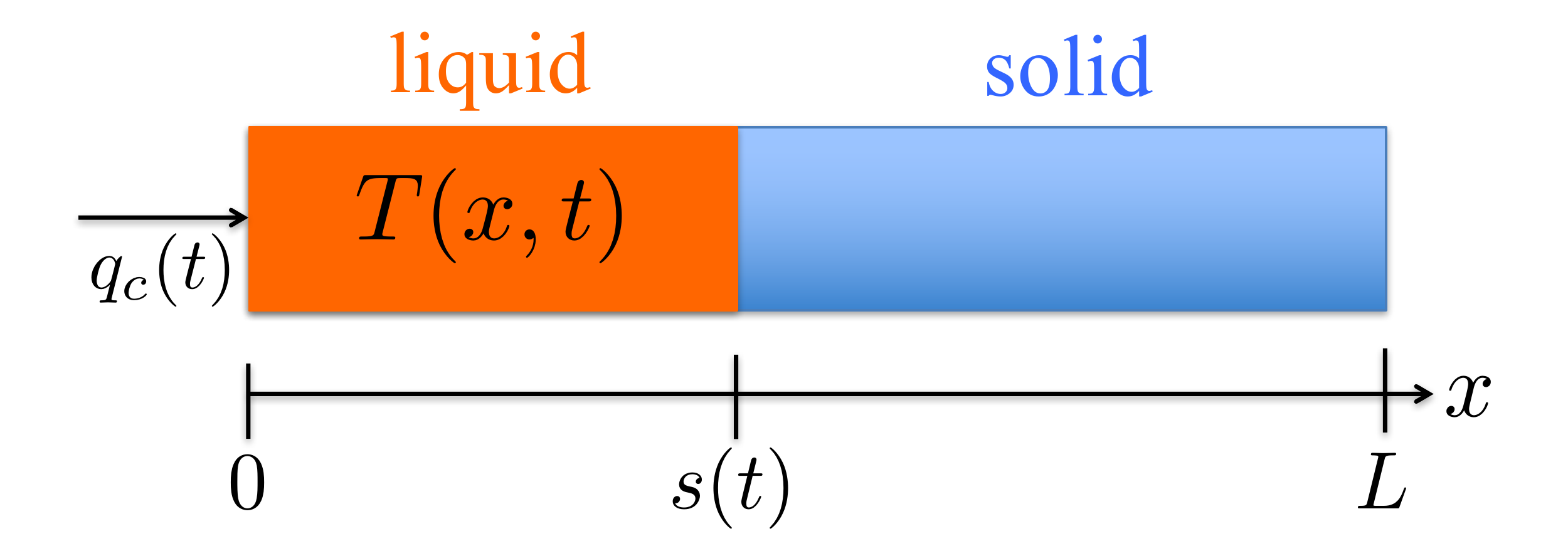}\\
\caption{Schematic of 1D Stefan problem.}
\label{fig:stefan}
\end{figure}
Consider a physical model which describes the melting or solidification mechanism in a pure one-component material of length $L$ in one dimension. In order to describe a position at which phase transition from liquid to solid occurs (or equivalently, in the reverse direction) mathematically, we divide the domain $[0, L]$ into the two time varying sub-domains, namely,  $[0,s(t)]$ occupied by the liquid phase, and $[s(t),L]$ by the solid  phase. A heat flux is manipulated into the material through the boundary at $x=0$ of the liquid phase, which affects the dynamics of the solid-liquid interface. As a consequence, the heat equation alone does not provide a complete description of the phase transition and must be coupled with an unknown dynamics that describes the moving boundary. Assuming that the temperature in the liquid phase is not lower than the melting temperature $T_{m}$ of the material, the following coupled system consisting of 
\begin{itemize}
\item the diffusion equation of the temperature in the liquid-phase which is written as
\begin{align}\label{eq:stefanPDE}
T_t(x,t)=\alpha T_{xx}(x,t), \hspace{1mm} 0\leq x\leq s(t), \hspace{1mm} \alpha :=\f{k}{\rho C_p}, 
\end{align}
with the boundary conditions
\begin{align}\label{eq:stefancontrol}
-k T_x(0,t)=q_{c}(t), \\ \label{eq:stefanBC}
T(s(t),t)=T_{m},
\end{align}
and the initial conditions
\begin{align}\label{eq:stefanIC}
T(x,0)=T_0(x), \quad s_{0} = s_0
\end{align}
where $T(x,t)$,  ${q}_c(t)$,  $\rho$, $C_p$ and  $k$ are the distributed temperature of the liquid phase, manipulated heat flux, liquid density, the liquid heat capacity, and the liquid heat conductivity, respectively.
\end{itemize}
and 
\begin{itemize}
\item the local  energy balance at the position of the liquid-solid interface $x=s(t)$  which yields to  the following ODE
\begin{align}\label{eq:stefanODE}
 \dot{s}(t)=-\beta T_x(s(t),t), \quad \beta :=\frac{k}{\rho \Delta H^*}
\end{align}
that describes the dynamics of moving boundary. The physical parameter  $\Delta H^*$ represents the latent heat of fusion.
\end{itemize}
For the sake of brevity, we refer the readers to  \cite{Gupta03}, where the Stefan condition in the case of a solidification process is derived. 
\begin{remark}\emph{
 As the moving interface  $s(t)$ is unknown explicitly, the problem defined in  \eqref{eq:stefanPDE}--\eqref{eq:stefanODE}  is a highly nonlinear
problem. Note that this non-linearity is purely geometric
rather than algebraic.}\end{remark}
\begin{remark}\label{assumption}\emph{
Due to the so-called isothermal interface condition
 that prescribes the melting temperature $T_{m}$ at the interface through  \eqref{eq:stefanBC},
this form of the Stefan problem is a reasonable model only if the following conditions hold:
\begin{align}
&T(x,t) \geq T_{m} \quad \textrm{ for all }\quad x \in [0,s(t)],\\
&\dot{s}(t)\geq0 \quad \textrm{ for all }\quad t\geq0 \label{moving}
\end{align}}
\end{remark}
From  Remark \ref{assumption}, it is plausible to assume the existence of a positive constant $H$ such that
\begin{align}\label{eq:stefanICbound}
0\leq T_0(x)-T_{m}\leq H(s_{0}-x).
\end{align}
We recall the following lemma that ensures the validity of the model \eqref{eq:stefanPDE}--\eqref{eq:stefanODE}.
\begin{lem}\label{lemma1}
For any $q_{c}(t)>0$ on the finite time interval $(0,\bar{t})$, $T(x,t)>T_{m}, ~\forall x\in(0,s(t))$ and $\forall t\in(0,\bar{t})$. And then $\dot{s}(t) >0$, $\forall t\in(0,\bar{t})$. 
\end{lem}
The proof of the lemma \ref{lemma1} is based on Maximum Principle and Hopf's Lemma as shown in \cite{Gupta03}. 

\section{Control Problem statement}\label{statement}
In this model, the heat flux $q_{c}(t)$ is manipulated as a boundary controller. The objective of control is to drive the moving boundary $s(t)$ to a desired position $s_{r}$ while ensuring the convergence of the ${\cal H}_1$-norm of the temperature $T(x,t)$ in the liquid phase. We denote the reference error of liquid temperature and the moving interface as $u(x,t)=T(x,t)-T_{m}$ and $X(t)=s(t)-s_{r}$, respectively. The main theorem of this paper is stated as follows.

\begin{thm}
Consider a closed-loop system consisting of the plant \eqref{eq:stefanPDE}--\eqref{eq:stefanODE} and the control law
\begin{align}\label{Fullcontrol}
q_c(t)=-ck \left(\frac{1}{\alpha}\int_0^{s(t)} (T(x,t)-T_m) dx \right.\nonumber\\
\left.+\frac{1 }{\beta}(s(t)-s_r)\right). 
\end{align}
where $c>0$ is an arbitral controller gain. Assume that the initial condition $(T_0(x), s_{0})$ is compatible with the control law and satisfies \eqref{eq:stefanICbound}. For any reference setpoint $s_r$ which satisfies the following inequality 
\begin{align}\label{compatibility}
s_r>s_{0}+\frac{C_p}{\Delta H^*}\int_{0}^{s_{0}} (T_0(x)-T_m) dx,
\end{align}
the closed-loop system is exponentially stable in the sense of the norm
\begin{align}
||T-T_m||_{{\cal H}_1}^2+(s(t)-s_r)^2.
\end{align}
The proof is to be given below.
\end{thm}

\section{Backstepping Transformation for Moving Boundary Formulation}\label{nonlineartarget}
\subsection{Direct transformation}
In this section, we introduce the following new backstepping transformation in order to achieve a target system that is exponentially stable. 
\begin{align}\label{eq:DBST}
w(x,t)=&u(x,t)-\frac{c}{\alpha} \int_{x}^{s(t)} (y-x)u(y,t) dy\nonumber\\
&+\frac{c}{\beta}(s(t)-x) X(t)
\end{align}
This transformation is an extension of the one first introduced in \cite{krstic2009}, to moving boundary problems proposed in \cite{Izadi15}. 
By taking the derivative of \eqref {eq:DBST} with respect to $t$ and $x$ respectively along the solution of \eqref{eq:stefanPDE}-\eqref{eq:stefanODE} with the control law \eqref{Fullcontrol}, we derive the following target system
\begin{align}\label{eq:tarPDE}
w_t(x,t)=\alpha w_{xx}(x,t)+\frac{c}{\beta}\dot{s}(t) X(t)
\end{align}
with the boundary conditions which are given as
\begin{align}\label{eq:tarBC1}
w(s(t),t) = 0\\
\label{eq:tarBC2}
w_x(0,t) = 0.
\end{align}
with the help of  \eqref {eq:DBST}, the ODE \eqref {eq:stefanODE} is rewritten as 
\begin{align}\label{eq:tarODE}
\dot{X}(t)=-cX(t)-\beta w_x(s(t),t)
\end{align}

\subsection{Inverse transformation}
The original system \eqref{eq:stefanPDE}--\eqref{eq:stefanODE} and the target system \eqref{eq:tarPDE}--\eqref{eq:tarODE} have equivalent stability properties if the transformation \eqref{eq:DBST} is invertible. 
Let us consider the following inverse transformation 
\begin{align}\label{eq:IBST}
u(x,t) = w(x,t) +\int_x^{s(t)} l(x,y)w(y,t) dy\nonumber\\
+h(s(t)-x)X(t),
\end{align}
where $l(x,y)$ is the kernel function. Taking derivative of \eqref{eq:IBST} with respect to $t$ and $x$ respectively along the solution of \eqref{eq:tarPDE}-\eqref{eq:tarODE}, the following relation is derived  
\begin{align}\label{time-deriv}
&u_t(x,t) -\alpha u_{xx}(x,t) = 2\alpha \left(\frac{d}{dx}l(x,x)\right)w(x,t)\nonumber\\
&-\alpha \int_x^{s(t)} (l_{xx}(x,y)-l_{yy}(x,y))w(y,t) dy\nonumber\\
&+\left(\frac{c}{\beta}\left(1+\int_x^{s(t)} l(x,y) dy\right)+h'(s(t)-x)\right) \dot{s}(t) X(t)\nonumber\\
&-\left(\alpha h''(s(t)-x)+ch(s(t)-x)\right)X(t)\nonumber\\
&+(\alpha l(x,s(t))-\beta h(s(t)-x))w_x(s(t),t)\\
&u(s(t),t)=h(0)X(t)\\
&u_x(s(t),t)=w_x(s(t),t)-h'(0)X(t)\label{eq:gradient}
\end{align}
In order to hold \eqref{eq:stefanPDE}-\eqref{eq:stefanODE} for any continuous functions $(w(x,t), X(t))$, by \eqref{time-deriv}-\eqref{eq:gradient}, $h(s(t)-x)$ and $l(x,y)$ satisfy 
\begin{align}
\label{eq:hode}
&\alpha h''(s(t)-x)+c h(s(t)-x)=0\\
\label{eq:hic1}
&h(0)=0, \quad h'(0)=-\frac{c}{\beta}\\
\label{eq:kernelpde}&l_{xx}(x,y)-l_{yy}(x,y)=0 \\
& \frac{d}{dx}l(x,x)=0\\
&\frac{c}{\beta}\left(1+\int_x^{s(t)} l(x,y) dy\right)+h'(s(t)-x)=0\\
\label{eq:kernelrel}&\alpha l(x,s(t))-\beta  h(s(t)-x)=0
\end{align}
By \eqref{eq:hode} and \eqref{eq:hic1}, the solution of $h(s(t)-x)$ is given by
\begin{align}\label{hxy}
h(s(t)-x )=-\frac{\sqrt{c\alpha}}{\beta}{\rm sin}\left(\sqrt{\frac{c}{\alpha}}(s(t)-x) \right)
\end{align}
In addition,  the conditions \eqref{eq:kernelpde}-\eqref{eq:kernelrel} hold for 
\begin{align}\label{lxy}
l(x,s(t))=-\sqrt{\frac{c}{\alpha}}{\rm sin}\left(\sqrt{\frac{c}{\alpha}}(s(t)-x)\right)
\end{align}
Finally, from \eqref{hxy} and \eqref{lxy}, the  inverse transformation is written as
\begin{align}
u(x,t) = &w(x,t)-\frac{\sqrt{c\alpha}}{\beta}{\rm sin}\left(\sqrt{\frac{c}{\alpha}}(s(t)-x) \right)X(t)\nonumber\\ & -\int_x^{s(t)} \sqrt{\frac{c}{\alpha}}{\rm sin}\left(\sqrt{\frac{c}{\alpha}}(y-x)\right)w(y,t) dy
\end{align}
\section{Physical constraints}\label{positivness}
Noting that $q_c(t)>0$ is required by Remark \ref{assumption} and Lemma \ref{lemma1}, the overshoot beyond the reference $s_r$ is prohibited to achieve the control objective $s(t) \to s_r$ due to its increasing property \eqref{moving}, which means $s(t)<s_r$ is required to be satisfied for $\forall t>0$. In this section, we derived  the condition that guarantees these two conditions 
\begin{align}\label{physical constraints}
q_c(t)>0, \quad  s(t)<s_r, \quad  \forall t>0
\end{align}
, namely "physical constraints". Physically, a negative controller may lead to a freezing process.\begin{prop}
If the initial condition satisfies \eqref{compatibility}, then $q_{c}(t)>0$ and $s_{0}<s(t)<s_r$ for $\forall t>0$.  
\end{prop}
\begin{proof}
By taking the time derivative of \eqref{Fullcontrol}, we have
\begin{align}
\dot{q}_{c}(t)=&-ck\left( \frac{1}{\alpha}\int_{0}^{s(t)} u_t(x,t) dx+\frac{1}{\alpha}\dot{s}(t)u(s(t),t)\right.\nonumber\\
&\left.+\frac{1}{\beta} \dot{X}(t)\right)=cku_x(0,t)=-cq_{c}(t)
\end{align}
The solution of $q_{c}(t)$ can be written explicitly as
\begin{align}
q_{c}(t)=q_{c}(0)e^{-ct}
\end{align}
Therefore, if the initial condition satisfies \eqref{compatibility}, which leads to $q_{c}(0)>0$, then we can state that $q_{c}(t)>0$ for any $t$. Next, by Lemma 1, if $q_{c}(t)>0$ for $\forall t>0$, then it deduces to 
\begin{align}\label{properties}
T(x,t)>T_m, \quad  \dot{s}(t)>0, \quad  \forall y\in (0,s(t)),  \forall t>0.
\end{align}
Knowing  $q_{c}(t)>0$ and $T(x,t)>T_m$, by \eqref{Fullcontrol}, we deduce
\begin{align}
s(t)<s_r,\quad \forall t>0\label{eq:error}
\end{align}
In addition, by $\dot{s}(t)>0$, it leads to $s_{0}< s(t)$. Connecting this result with \eqref{eq:error}, we derive 
\begin{align}\label{position}
s_{0}<s(t)<s_r, \quad \forall t>0.  
\end{align}
\end{proof}
In the following, we assume that \eqref{compatibility} is satisfied. We use the property of \eqref{properties} and \eqref{position} to show the Lyapunov stability. 

\section{Lyapunov Stability}\label{stability}
We recall that in order to design backstepping controller \eqref{Fullcontrol} and the target system \eqref{eq:tarPDE}--\eqref{eq:tarODE}   we introduce the transformation \eqref{eq:DBST} and its inverse \eqref{eq:IBST}. The exponential  stability of the scaled system \eqref{eq:stefanPDE}--\eqref{eq:stefanODE} is guaranted if  the nonlinear target system  \eqref{eq:tarPDE}--\eqref{eq:tarODE} is  exponentially stable and the  transformation \eqref{eq:DBST} admits a unique inverse defined as \eqref{eq:IBST}. In the following we establish the exponential stability of  the closed-loop control system $ H^1$-norm of the temperature and the moving boundary  based on the analysis of the associated target system \eqref{eq:tarPDE}--\eqref{eq:tarODE} .

\subsection{Exponential stability for the target $(w,X)$-system}
Let $V_{1}$ be a functional such that
\begin{align}\label{eq:lyapunov}
V_{1} = &\frac{1}{2}\int_0^{s(t)} w(x,t)^2 dx\nonumber\\
&+ \frac{1}{2}\int_0^{s(t)} w_x(x,t)^2 dx+\frac{p}{2}X(t)^2
\end{align}
with a positive number $p>0$ which is chosen later. Then, by taking the derivative of \eqref{eq:lyapunov} along the solution of the target system \eqref{eq:tarPDE}-\eqref{eq:tarODE} , we have the following
\begin{align}\label{eq:candidate}
\frac{dV_{1}}{dt}&=-\alpha\int_0^{s(t)} w_{xx}(x,t)^2 dx\nonumber\\&-\alpha\int_0^{s(t)} w_{x}(x,t)^2 dx-pcX(t)^2-p\beta X(t) w_x(s(t),t)\nonumber\\
&+\dot{s}(t)\left(\alpha\frac{c}{\beta} X(t) \int_0^{s(t)} w(x,t) dx \right. \nonumber\\ &\left. -\alpha\frac{c}{\beta} X(t)w_{x}(s(t),t)-\frac{1}{2}w_x(s(t),t)^2\right)
\end{align}
By \eqref{position}, Pointcare's inequality and Agmon's inequality are obtained as
\begin{align}
\int_0^{s(t)}w(x,t)^2 dx\leq 4s_{r}^{2}\int_0^{s(t)}w_{x}(x,t)^2 dx , \\
w_x(s(t),t)^2\leq 4s_{r}\int_0^{s(t)}w_{xx}(x,t)^2 dx .
\end{align}
Applying these, Young's inequality, Cauchy-Schwartz inequality, and \eqref{properties} to \eqref{eq:candidate}, we have
\begin{align}
&\frac{dV_{1}}{dt}\nonumber\\
&\leq-\frac{\alpha}{4s_{r}^{2}+1}\left(\int_0^{s(t)} w_{x}(x,t)^2 dx+\int_0^{s(t)} w(x,t)^2 dx\right)\nonumber\\
&-p\left(c-\frac{p\beta^2s_{r}}{\alpha}\right)X(t)^2\nonumber\\
&+\frac{c\alpha\dot{s}(t)}{2\beta}\left( \int_0^{s(t)} w(x,t)^2 dx+ \left(1+\frac{c\alpha}{\beta}\right)X(t)^2\right)
\end{align}
Suppose that $p$ is chosen to satisfy the following
\begin{align}\label{eq:gain}
0<p<\frac{c\alpha}{\beta^2s_{r}}
\end{align}
For any choice of $c>0$, there exists $p>0$ such that \eqref{eq:gain} holds. Consider the Lyapunov functional
\begin{align}\label{eq:lyap}
V=V_{1}e^{-as(t)}
\end{align}
then, by choosing the parameter $a$ such that
\begin{align}
a=\frac{c\alpha}{2\beta}{\rm max}\left\{1, \frac{1}{p}\left(1+\frac{c\alpha}{\beta}\right)\right\}, 
\end{align}
we have
\begin{align}
\frac{dV}{dt}&\leq\left[-\frac{\alpha}{4s_{r}^{2}+1}\left(\int_0^{s(t)} w_{x}(x,t)^2 dx\right.\right.\nonumber\\ &+\left.\left.\int_0^{s(t)} w(x,t)^2 dx\right)-p\left(c-\frac{p\beta^2s_{r}}{\alpha}\right)X(t)^2\right]e^{-as(t)}\nonumber\\
&\leq -bV
\end{align}
where $b$ is defined as
\begin{align}
b={\rm min}\left\{\frac{2\alpha}{4s_{r}^{2}+1}, 2\left(c-\frac{p\beta^2s_{r}}{\alpha}\right)\right\}
\end{align}
By the relation \eqref{eq:lyap} and Corollary 1 that $s_{0}\leq s(t) \leq s_{r}$, we arrive at
\begin{align}\label{eq:exp}
V_{1}(t)\leq V_{1}(0)e^{a(s_{r}-s_{0})}e^{-bt}
\end{align}

\subsection{Exponential stability for the original $(u, X)$-system}
From the direct transformation (\ref{eq:DBST}) and the inverse transformation (\ref{eq:IBST}) and using Young's and Cauchy-Schwarz inequality, we have
\begin{align}
\hspace{-1mm}\int_0^{s(t)}w(x,t)^2 dx & \leq M_1\int_0^{s(t)}u(x,t)^2dx+ M_2X(t)^2\label{int}
\end{align}
and
\begin{align}
\int_0^{s(t)}&w_x(x,t)^2 dx \leq 3\int_0^{s(t)}u_x(x,t)^2dx \nonumber\\ &+ M_3\int_0^{s(t)}u(x,t)^2dx +M_4X(t)^2\label{int_x}
\end{align}
\begin{align}
\hspace{-1mm}\int_0^{s(t)}u(x,t)^2  dx &\leq M_5\int_0^{s(t)}w(x,t)^2dx +  M_6X(t)^2\label{int2}
\end{align}
\begin{align}
&\int_0^{s(t)} u_x(x,t)^2  dx  \leq 3\int_0^{s(t)}w_x(x,t)^2dx  \nonumber\\ &+  M_7\int_0^{s(t)} w(x,t)^2 dx+M_8X(t)^2\label{int_x2}
\end{align}

where
\begin{align}
M_1=3\left(1+\frac{c^2s_{r}^{3}}{3\alpha^2}\right), \quad M_2= \frac{c^2s_{r}^{3}}{\beta^2}\\
M_3=3\frac{c^2s_{r}}{\alpha^2}, \quad M_4=3\frac{c^2s_{r}}{\beta^2}\\
M_5=3\left(1+\frac{cs_{r}^{2}}{2\alpha}\left\{1-{\rm sinc}\left(2\sqrt{\frac{c}{\alpha}}s_{r}\right)\right\}\right)\\
M_6=\frac{3c\alpha s_{r}}{2\beta^2}\left\{1-{\rm sinc}\left(2\sqrt{\frac{c}{\alpha}}s_{r}\right)\right\}\\
M_7=\frac{3c^2 s_{r}^{2}}{2\alpha^2}\left\{1+{\rm sinc}\left(2\sqrt{\frac{c}{\alpha}}s_{r}\right)\right\}\\
M_8=\frac{3c^2 s_{r}}{2\beta^2}\left\{1+{\rm sinc}\left(2\sqrt{\frac{c}{\alpha}}s_{r}\right)\right\}
\end{align}
Adding \eqref{int} to \eqref{int_x} and \eqref{int2} to \eqref{int_x2}, we derive the following inequality
\begin{align}\label{eq:bound}
&\underline{\delta}\left(\int_0^{s(t)}u(x,t)^2 dx+\int_0^{s(t)}u_x(x,t)^2dx+X(t)^2\right)\nonumber\\
&\hspace{-2mm}\leq\int_0^{s(t)}w(x,t)^2 dx+\int_0^{s(t)}w_x(x,t)^2 dx +pX(t)^2\nonumber\\
&\hspace{-2mm}\leq \bar{\delta}\left(\int_0^{s(t)}u_x(x,t)^2dx+\int_0^{s(t)}u(x,t)^2dx+X(t)^2\right)
\end{align}
where
\begin{align}
&\bar{\delta}={\rm max}\{M_1+M_3, p+M_2+M_4\}\nonumber\\
&\underline{\delta}=\frac{{\rm min}\left\{1,p\right\}}{{\rm max}\left\{M_5+M_7, M_6+M_8+1\right\}}
\end{align}
Defining the parameter $D>0$ as
\begin{align}
D=\frac{\bar{\delta}}{\underline{\delta}}e^{as_r},
\end{align}
by \eqref{eq:exp} and \eqref{eq:bound}, for any choice of $c>0$ there exists $D>0$ and $b>0$ such that
\begin{align}
&\int_0^{s(t)}u(x,t)^2 dx+\int_0^{s(t)}u_x(x,t)^2dx+X(t)^2\nonumber\\
&\leq D\left(\int_0^{s_{0}}u_0(x)^2 dx+\int_0^{s_{0}}u_{0}'(x)^2dx+X(0)^2\right)e^{-bt},
\end{align}
which completes the proof of Theorem 1. 

\section{Numerical Simulation}\label{simulation}
	As in \cite{maidi2014}, the simulation is performed considering a strip of zinc whose physical properties are given in Table 1 using the well known boundary immobilization method and finite difference semi-discretization. The setpoint of the interface is $s_{r}$=0.35m. The initial distribution of the temperature error is set as $T_0(x)-T_m=H(s_{0}-x)$ with $H$=10000K$\cdot {\rm m}^{-1}$. The controller gain $c$ is chosen arbitrarily, but small enough to avoid numerical instabilities, and here it is chosen $c$=0.01. The dynamics of the moving interface $s(t)$ and ${\cal H}_1$-norm of temperature error $||T(x,t)-T_m||_{{\cal H}_1}$ are depicted in Fig. \ref{fig:interface} and Fig. \ref{fig:h1}, respectively for two different initial position of the moving interface, namely, $s_{0}$=0.01m (blue dash) and $s_{0}$=0.25m (red dash). Time evolution of the control input  is depicted in Fig \ref{fig:control}. The simulation of coupled system with $s_{0}$ =0.01m shows that the interface converges to its setpoint while keeping $\dot{s}(t)>0$ and $s(t)<s_r$ with a positive control signal $q_{c}(t)>0$ as we expected from theoretical result, because the setpoint and initial condition satisfy \eqref{compatibility}. However, the system with the interface initialized at the position $s_{0}$ =0.25m leads to $\dot{s}(t)<0$,  $s(t)>s_r$, and a negative  control signal because the choice of setpoint doesn't satisfy \eqref{compatibility}. Therefore, the numerical simulation is consistent with our theoretical result. We emphasize that the proposed controller does not require the restriction imposed in \cite{maidi2014}  regarding the material properties, although the equation of the controller is the same as the one proposed in  \cite{petrus2010} for a Stefan problem which describes a solidification process. In that sense,  the proposed controller in the present result  offers more modularity to a large class of materials, and guarantees the exponential stability of the sum of interface error and ${\cal H}_1$-norm of temperature error, compared to \cite{petrus2010} which provides only an asymptotical stability result. 

\begin{table}[htb]
\caption{Physical properties of zinc}
\begin{center}
    \begin{tabular}{| l | l | l | }
    \hline
    $\textbf{Description}$ & $\textbf{Symbol}$ & $\textbf{Value}$ \\ \hline
    Density & $\rho$ & 6570 ${\rm kg}\cdot {\rm m}^{-3}$\\ 
    Latent heat of fusion & $\Delta H^*$ & 111,961${\rm J}\cdot {\rm kg}^{-1}$ \\ 
    Heat Capacity & $C_p$ & 389.5687 ${\rm J} \cdot {\rm kg}^{-1}\cdot {\rm K}^{-1}$  \\  
    Thermal conductivity & $k$ & 116 ${\rm w}\cdot {\rm m}^{-1}$  \\ \hline
    \end{tabular}
\end{center}
\end{table}

\begin{figure}[htb]
\centering
\includegraphics[width=3.2in]{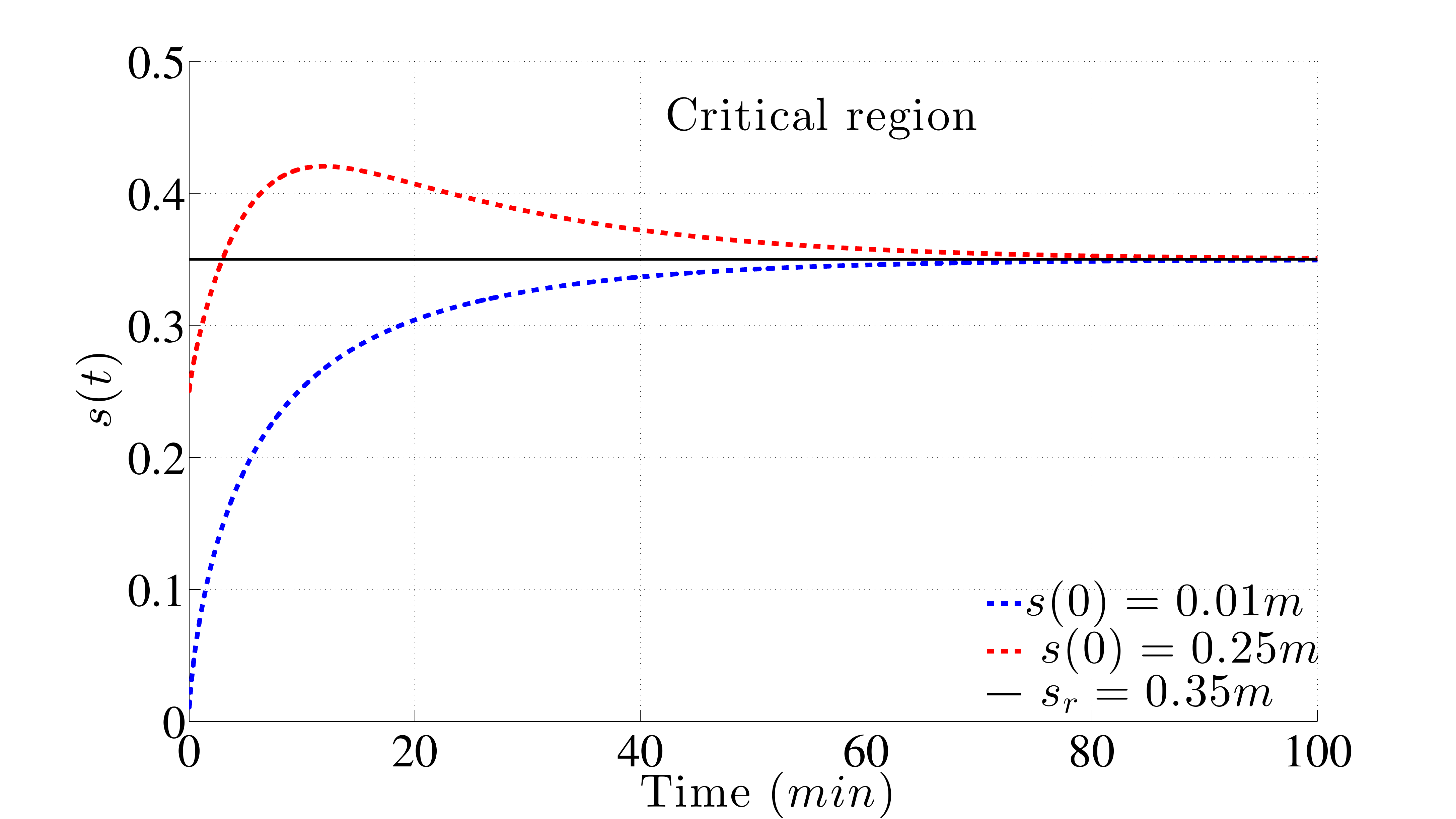}
\caption{The moving interface.}
\label{fig:interface}
\end{figure}

\begin{figure}[htb]
\centering
\includegraphics[width=3.2in]{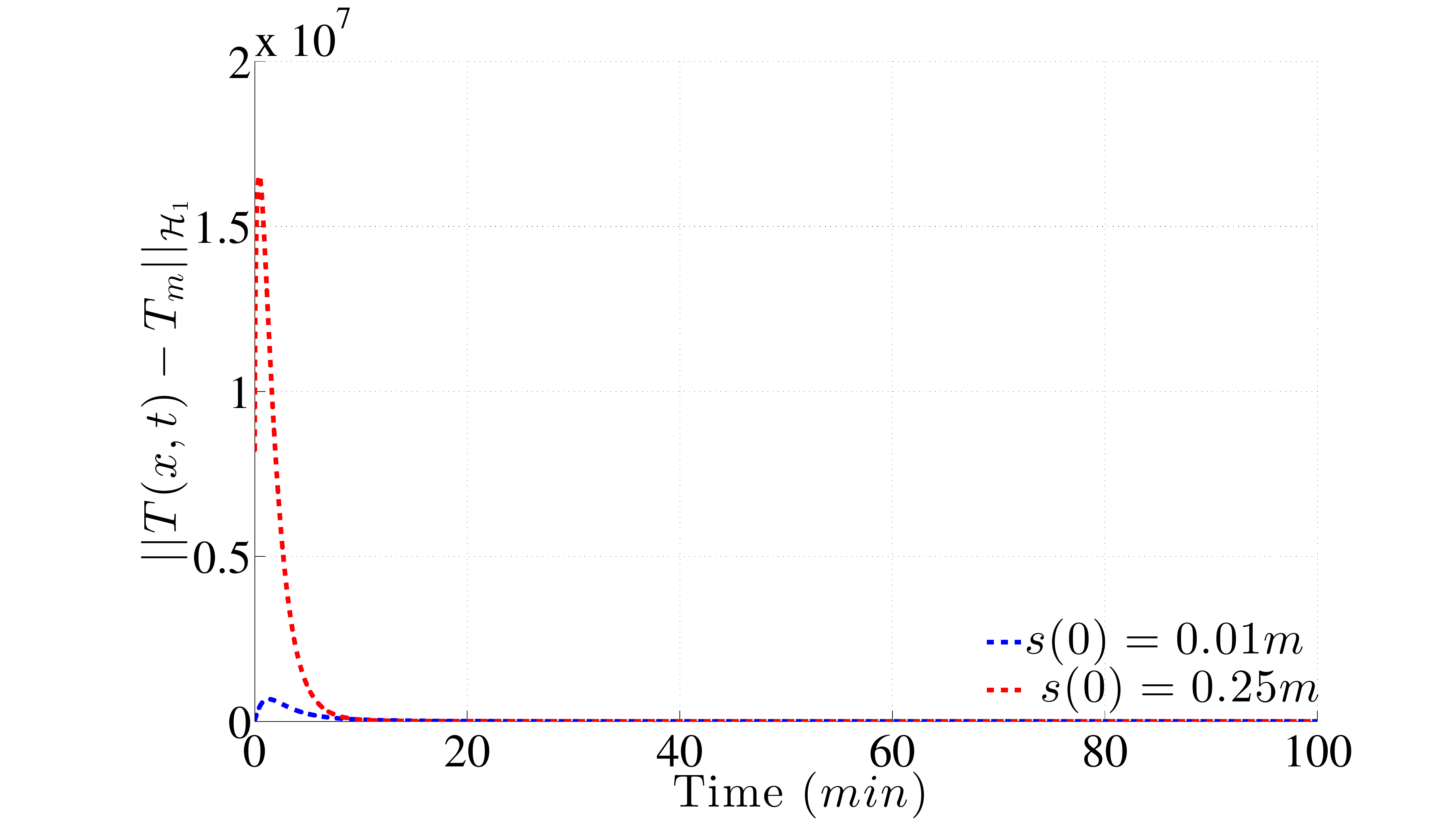}
\caption{${\cal H}_1$-norm of the temperature.}
\label{fig:h1}
\end{figure}

\begin{figure}[htb]
\centering
\includegraphics[width=3.2in]{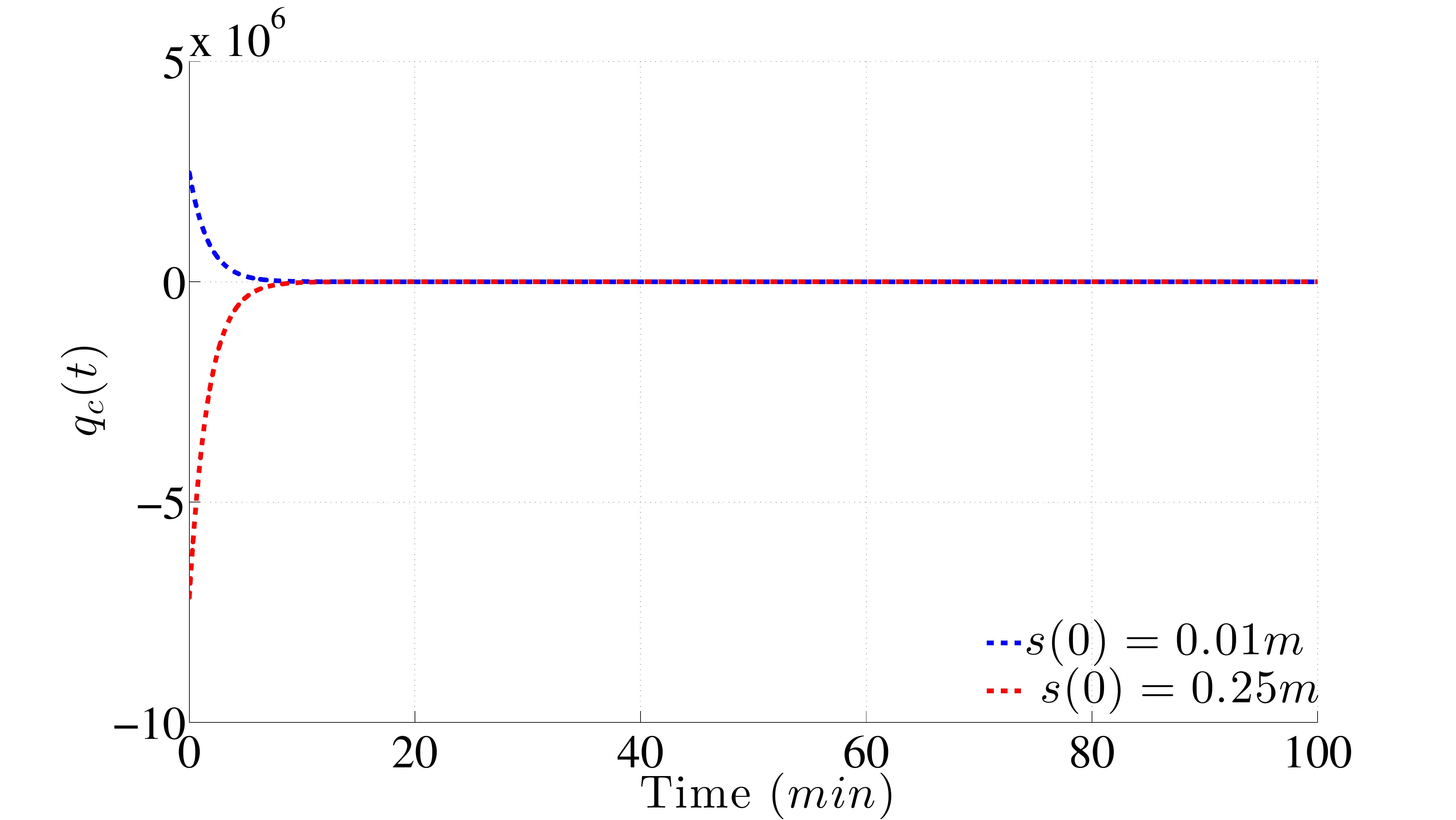}
\caption{The positiveness verification of the controller.}
\label{fig:control}
\end{figure}

\section{Conclusions and Future works}\label{conclusion}
Along this paper we studied  a one-phase Stefan problem in 1-D and proposed a boundary feedback controller that achieves the exponential stability of sum of the moving interface and the ${\cal H}_1$-norm of the temperature based on the full state measurement. A nonlinear backstepping transformation for moving boundary problem is utilized and the controller is proved to remain positive, which guarantees some physical properties required for the validity of the model and the proof of stability. There are two contributions stated in our conclusion. Firstly, our approach offers an interesting perspective regarding the backstepping control of moving boundary problem whose dynamics depends on the system. Secondly, we showed the exponential stability of the sum of the interface error and the ${\cal H}_1$-norm of temperature error, while in \cite{maidi2014} and \cite{petrus2010} it was shown asymptotical stability of interface error and exponential stability of ${\cal L}_2$-norm of temperature error.  Adding a dissipative term in the transformation would allow faster convergence and must be considered as a future direction. The design of an observer that enables to reconstruct the full state based on the boundary measurement is of practical interest in such type of problem and will be considered in our future work.

\bibliographystyle{unsrt}
\bibliography{ref}

\begin{thebibliography}{10}

\bibitem{wett91}
J.S. Wettlaufer.
\newblock Heat flux at the ice-ocean interface.
\newblock {\em Journal of Geophysical Research}, 96(C4):297--313, 1991.

\bibitem{petrus2012}
B.~Petrus, J.~Bentsman, and B.G. Thomas.
\newblock Enthalpy-based feedback control algorithms for the stefan problem.
\newblock In {\em CDC}, pages 7037--7042, 2012.

\bibitem{conrad_90}
F.~Conrad, D.~Hilhorst, and T.~I. Seidman.
\newblock Well-posedness of a moving boundary problem arising in a
  dissolution-growth process.
\newblock {\em Nonlinear Analysis}, 15(5):445 -- 465, 1990.

\bibitem{Belen03}
B.~Zalba, J.M. Marin, L.F. Cabeza, and H.~Mehling.
\newblock Review on thermal energy storage with phase change: materials, heat
  transfer analysis and applications.
\newblock {\em Applied Thermal Engineering}, 23(3):251 -- 283, 2003.

\bibitem{Daraoui2010}
N.~Daraoui, P.~Dufour, H.~Hammouri, and A.~Hottot.
\newblock Model predictive control during the primary drying stage of
  lyophilisation.
\newblock {\em Control Engineering Practice}, 18(5):483--494, 2010.

\bibitem{Armaou01}
A.~Armaou and P.D. Christofides.
\newblock Robust control of parabolic {PDE} systems with time-dependent spatial
  domains.
\newblock {\em Automatica}, 37(1):61 -- 69, 2001.

\bibitem{Petit10}
N.~Petit.
\newblock Control problems for one-dimensional fluids and reactive fluids with
  moving interfaces.
\newblock In {\em Advances in the theory of control, signals and systems with
  physical modeling}, volume 407 of {\em Lecture notes in control and
  information sciences}, pages 323--337, Lausanne, Dec 2010.

\bibitem{Christofides98_Parabolic}
Panagiotis~D. Christofides.
\newblock Robust control of parabolic {PDE} systems.
\newblock {\em Chemical Engineering Science}, 53(16):2949 -- 2965, 1998.

\bibitem{maidi2014}
A.~Maidi and J.-P. Corriou.
\newblock Boundary geometric control of a linear stefan problem.
\newblock {\em Journal of Process Control}, 24(6):939--946, 2014.

\bibitem{karvaris1990-a}
C.~Karvaris and J.C. Kantor.
\newblock Geometric methods for nonlinear process control i.
\newblock {\em Background, Industrial \& Engineering Chemistry Research},
  29:2295--2310, 1990.

\bibitem{karvaris1990-b}
C.~Karvaris and J.C. Kantor.
\newblock Geometric methods for nonlinear process control ii.
\newblock {\em Controller synthesis, Industrial \& Engineering Chemistry
  Research}, 29:2310--2323, 1990.

\bibitem{maidi2009}
A.~Maidi, M.~Diaf, and J.-P. Corriou.
\newblock Boundary geometric control of a counter-current heat exchanger.
\newblock {\em Journal of Process Control}, 19(2):297--313, 2009.

\bibitem{krstic2008boundary}
M.~Krstic and A.~Smyshlyaev.
\newblock {\em Boundary control of PDEs: A course on backstepping designs},
  volume~16.
\newblock Siam, 2008.

\bibitem{andrew2004}
A.~Smyshlyaev and M.~Krstic.
\newblock Closed-form boundary state feedbacks for a class of 1-d partial
  integro-differential equations.
\newblock {\em Automatic Control, IEEE Transactions on}, 49(12):2185--2202, Dec
  2004.

\bibitem{petrus2010}
Bryan Petrus, Joseph Bentsman, and Brian~G Thomas.
\newblock Feedback control of the two-phase stefan problem, with an application
  to the continuous casting of steel.
\newblock In {\em Decision and Control (CDC), 2010 49th IEEE Conference on},
  pages 1731--1736. IEEE, 2010.

\bibitem{Izadi15}
M.~Izadi and S.~Dubljevic.
\newblock Backstepping output-feedback control of moving boundary parabolic
  {PDE}s.
\newblock {\em European Journal of Control}, 21(0):27 -- 35, 2015.

\bibitem{krstic2009}
M.~Krstic.
\newblock Compensating actuator and sensor dynamics governed by diffusion
  {PDE}s.
\newblock {\em Systems \& Control Letters}, 58(5):372--377, 2009.

\bibitem{Gian2011}
G.A. Susto and M.~Krstic.
\newblock Control of {PDE}--{ODE} cascades with neumann interconnections.
\newblock {\em Journal of the Franklin Institute}, 347(1):284--314, 2010.

\bibitem{tang2011state}
S.~Tang and C.~Xie.
\newblock State and output feedback boundary control for a coupled {PDE}--{ODE}
  system.
\newblock {\em Systems \& Control Letters}, 60(8):540--545, 2011.

\bibitem{Gupta03}
S.~Gupta.
\newblock {\em The classical Stefan problem. Basic concepts, Modelling and
  Analysis}.
\newblock Applied mathematics and Mechanics. North-Holland, 2003.

\end{thebibliography}

\end{document}